\documentclass[10pt]{amsart}

\usepackage{amssymb}
\usepackage{soul}
\usepackage{bm}
\usepackage{graphicx}
\usepackage{psfrag}
\usepackage[usenames,dvipsnames]{xcolor}
\usepackage{enumerate}
\usepackage{multirow}
\usepackage{url}
\usepackage{hyperref}

\usepackage{subcaption}
\usepackage{comment}

\usepackage{algpseudocode}

\usepackage{mathtools}

\usepackage{xy}
\input xy
\xyoption{all}

\usepackage{tikz}
\usetikzlibrary{arrows}

\newcommand{\excise}[1]{}%{$\star$\textsc{#1}$\star$}
\newtheorem{thm}{Theorem}[section]
\newtheorem{lemma}[thm]{Lemma}

\newtheorem{cor}[thm]{Corollary}

\theoremstyle{definition}
\newtheorem{alg}[thm]{Algorithm}

\newtheorem{defn}[thm]{Definition}

\numberwithin{equation}{section}

\renewcommand\>{\rangle}

\newcommand\NN{\mathbb{N}}

\newcommand\cP{{\mathcal P}}
\newcommand\powerset{{\cP}}

\newcommand\abs[1]{|#1|}

\newcommand{\quotes}[1]{{``#1"}}
\newcommand{\Zz}{\mathsf Z}
\newcommand{\incr}{\mathsf{incr}}
\renewcommand\aa{{\mathbf a}}
\newcommand\ee{{\mathbf e}}
\newcommand{\kernelcall}[1]{<\hspace{-0.4em}<\hspace{-0.4em}<#1>\hspace{-0.4em}>\hspace{-0.4em}>}
\let\oldquote\quote
\let\endoldquote\endquote

\RenewDocumentEnvironment{quote}{om}
  {\oldquote}
  {\par\nobreak\smallskip
   \hfill(#2\IfValueT{#1}{~---~#1})\endoldquote 
   \addvspace{\bigskipamount}}
\begin{document}%%%%%%%%%%%%%%%%%%%%%%%%%%%%%%%%%%%%%%%%%%%%%%%%%%%%%%%%
%%%%%%%%%%%%%%%%%%%%%%%%%%%%%%%%%%%%%%%%%%%%%%%%%%%%%%%%%%%%%%%%%%%%%%%%

\mbox{}
%\vspace{-2ex}%-1.1743pt}
\title
[Two parallel dynamic lexicographic factorization algorithms]
{Two parallel dynamic lexicographic algorithms for factorization sets in numerical semigroups}
\author[Thomas Barron]{Thomas Barron}
\address{Denver, CO 
}
\email{t@tbarron.xyz}

\date{\today}

\begin{abstract}%\normalsize
To the existing dynamic algorithm \textsf{FactorizationsUpToElement} for factorization sets of elements in a numerical semigroup, we add lexicographic and parallel behavior. To the existing parallel lexicographic algorithm for the same, we add dynamic behavior.
The (dimensionwise) dynamic algorithm is parallelized either elementwise or factorizationwise, while the parallel lexicographic algorithm is made dynamic with low-dimension tabulation.
The tabulation for the parallel lexicographic algorithm can itself be performed using the dynamic algorithm.
We provide reference CUDA implementations with measured runtimes.
\end{abstract}
\maketitle
\vspace{-1em}
% \setcounter{tocdepth}{1}
% \tableofcontents
%%%%%%%%%%%%%%%%%%%%%%%%%%%%%%%%%%%%%%%%%%%%%%%%%%%%%%%%%%%%%%%%%%%%%%%%%
\section{Introduction}%%%%%%%%%%%%%%%%%%%%%%%%%%%%%%%%%%%%%%%%%%%%%%%%%%%
\label{sec:intro}%%%%%%%%%%%%%%%%%%%%%%%%%%%%%%%%%%%%%%%%%%%%%%%%%%%%%%%%
%raggedbottom%%%%%%%%%%%%%%%%%%%%%%%%%%%%%%%%%%%%%%%%%%%%%%%%%%%%%%%%%%%%

The \textbf{factorization set} of an element $n \in \NN$ with respect to a set of generators $(g_1,...,g_d) \in \NN^d$ is defined as:

\begin{minipage}{0.45\textwidth}
\[\Zz : \NN \times \NN^d \to \powerset(\NN^d)\]
\end{minipage}
\begin{minipage}{0.45\textwidth}
\[\Zz\left(n, (g_1, ..., g_d)\right) := \left\{(a_1,...,a_d) \middle| \sum_i a_ig_i = n\right\}\]
\end{minipage}

The generators may be omitted to write $\Zz(n)$ when they remain clear from context.

\textbf{Dynamic programming} involves taking advantage of \emph{overlapping subproblems}, meaning that the problem can be decomposed into subproblems which reoccur more than once. An existing dynamic algorithm \textsf{FactorizationsUpToElement} \cite{dynamic} takes advantage of a \emph{dimensionwise} recurrence relation, specifically $\Zz(n) = \bigcup_i \incr_i(\Zz(n-g_i))$. In this work we show that a variant of this algorithm already proceeds in lexicographic order, and that the algorithm can be parallelized, either elementwise, by computing in batches of $g_{smallest}$ and assigning each of $\Zz(x+1), ..., \Zz(x+g_{smallest})$ to different workers, or factorizationwise, by assigning the many $\mathsf{incr}_i$ copy-and-increment operations to different workers.

Parallel bounded lexicographic enumeration \cite{parallelboundedlexicographic} of the factorization set of an element
assigns each worker a lexicographically bounded slice of the final output, and each worker returns zero or one factorizations per iteration of \textsf{nextCandidate}.
In this work we add dynamic behavior via \emph{low-dimension tabulation} (or \emph{memoization}),
in which the fixing of some number of leading coordinates leads to a lexicographically contiguous subproblem on the remaining final coordinates, the solutions to which may be looked up in the table (or \emph{memo}) and all returned at once.

We update the reference CUDA \cite{cuda} implementation of parallel bounded lexicographic streams from \cite{parallelboundedlexicographic} to include the above dynamic behavior;
we implement the lexicographic variant of the dynamic algorithm \textsf{FactorizationsUpToElement} from \cite{dynamic} with the above parallelism;
and we utilize the dynamic algorithm as an initial set-up step to populate the
memo
for the parallel lexicographic algorithm.
We measure runtimes in a variety of scenarios.

\section{The (dimensionwise) dynamic factorizations algorithm}

The dynamic algorithm \textsf{FactorizationsUpToElement} in \cite{dynamic} computes $\Zz(n)$ by computing \linebreak$\Zz(0), \Zz(1), ..., \Zz(n)$, where each step utilizes the recurrence relation provided by 
%looking backwards by $1$ in each index in the factorization space, which under the image of $\phi:\NN^d \to \NN$ corresponds to 
looking backwards by each $g_i$ and incrementing the respective factorization sets thereof by 1 in index $i$.
\begin{defn}[Incrementing one index of a factorization]\ \\
\begin{minipage}{0.45\textwidth}
$$\incr_i: \Zz(n) \to \Zz(n+g_i)$$
\end{minipage}
\begin{minipage}{0.45\textwidth}
$$(a_1,...,a_i,...,a_d) \mapsto (a_1,...,a_i+1,...,a_d)$$
\end{minipage}
\end{defn}

\begin{lemma}[{The dynamic recurrence relation for factorizations; \cite[Lemma 3.1]{dynamic}}]\label{l:nondisjointrecurrence}
$$\Zz(n) = \bigcup_i \incr_i(\Zz(n-g_i))$$
\end{lemma}

There is another form of this expression as a disjoint union, which we may find useful later.

\begin{lemma}[{Dynamic recurrence relation as disjoint union; \cite[Lemma 3.1]{dynamic}}]\label{disjdyn}
$$\Zz(n)=\bigsqcup_i\incr_i\left(\left\{\vec a \in \Zz(n-g_i) \mid {a_j = 0\ \forall j < i}\right\}\right)$$
\end{lemma}

We can use a shorthand for the inner expression of Lemma \ref{disjdyn}.

\begin{defn}[Factorizations \emph{not} supported left of index $i$]
$$\Zz_{\ge i}(n) := \left\{\vec a \in \Zz(n) \mid {a_j = 0\ \forall j < i}\right\}$$
\end{defn}

Using this notation Lemma \ref{disjdyn} becomes, more concisely,
$$\Zz(n)=\bigsqcup_i\incr_i\left(\Zz_{\ge i}(n-g_i)\right)$$

For reference, we reproduce in full the algorithm \textsf{FactorizationsUpToElement} as it originally appeared in \cite{dynamic}, before making some small modifications. $\ee_i$ is the unit vector in coordinate $i$.

\begin{alg}[\cite{dynamic}, Algorithm 3.3]\label{a:dynamicfactor}
Given $n \in S = \<n_1, \ldots, n_k\>$, computes $\mathsf Z(m)$  for all $m \in [0, n] \cap S$.  
\begin{algorithmic}
\Function{FactorizationsUpToElement}{$S$, $n$}
\State $F[0] \gets \{\mathbf 0\}$
\ForAll{$m \in [0, n] \cap S$}
	\State $Z \gets \{\}$
	\ForAll{$i = 1, 2, \ldots, k$ with $m - n_i \in S$}
		\State $Z \gets Z \cup \{\aa + \ee_i  : \aa \in F[m - n_i]\}$
	\EndFor
	\State $F[m] \gets Z$
\EndFor
\State \Return $F$
\EndFunction
\end{algorithmic}
\end{alg}

The evaluation of the expression \quotes{$x \in S$} relies on some (pre-calculated) notion of membership, often implemented as storing, for each $0 \le i < g_{smallest}$, the lowest number $n_i$ which has $n_i \equiv i \bmod g_{smallest}$. If one wishes to avoid this setup step, the conditions \quotes{$m \in [0,n] \cap S$} and \quotes{$m-n_i \in S$} can be relaxed to \quotes{$m \in [0,n]$} and \quotes{$m-n_i \ge 0$}: for if $m \not\in S$, then $m-n_i \not\in S$ for all $i$,
% meaning $\Zz(m-n_i) = \varnothing$,
so the algorithm will correctly produce $\Zz(m) = \varnothing$; and if $m-n_i \not\in S$, then $\Zz(m-n_i) = \varnothing$ and thus $\incr_i(\Zz(m-n_i)) = \varnothing$. This would allow the algorithm to return not just $\mathsf Z(m)$  for all $m \in [0, n] \cap S$ but instead for all $[0,n]$.

\subsection{Lexicographic order}
Algorithm \ref{a:dynamicfactor} is written using \emph{sets} as the storage data structure and \emph{set union} as the relevant combination operation. Sets are desirable to directly utilize in mathematical programming as most mathematical definitions begin in the realm of set theory (ZFC, constructive, or otherwise). In popular programming languages such as Python \cite{python}, sets are implemented as \emph{hash sets}, where uniqueness of members is maintained via a \emph{hash table}.

An alternative to using sets with set union, is to use \emph{lists} with \emph{concatenation} (for which we will use the same notation as disjoint union on sets, $\sqcup$). Lists are inherently ordered, which is desirable if one wishes to maintain a certain order on the factorizations, perhaps lexicographic.

\begin{lemma}\label{disjlex}
% Denote by $\Zz_{lex}(n)$ the \emph{lexicographically ordered} (descending) factorization set of $n$.
The indexing in Lemma \ref{disjdyn} respects lexicographic order:
$$\Zz(n)=
\incr_1\left(\Zz_{\ge 1}(n-g_1)\right)\ \bigsqcup^>\ ...\ \bigsqcup^>\ 
\incr_i\left(\Zz_{\ge i}(n-g_i)\right)\ \bigsqcup^>\ ...\ \bigsqcup^>\ 
\incr_d\left(\Zz_{\ge d}(n-g_d) \right)$$

\end{lemma}
\begin{proof}
    Each $\incr_i(\Zz_{\ge i}(n-g_i))$ is setwise greater than the next, as $\vec a \in \incr_i(\Zz_{\ge i}(n-g_i))$ has $a_i>0$ by virtue of being incremented by $\incr_i$, and $\vec b \in \incr_{i+1}(\Zz_{\ge {i+1}}(n-g_{i+1}))$ has $b_i=0$ by virtue of $\Zz_{\ge i+1}(n-g_{i+1})$ having zeroes left of coordinate $i+1$.
\end{proof}

\begin{cor}\label{resultislex}
    Interpreting $\bigsqcup$ in Lemma \ref{disjlex} as concatenation of lists, if each $\Zz(n-g_i)$ is already stored in lexicographic order, the resulting $\Zz(n)$ will also be in lexicographic order.
\end{cor}

We now present the natural variant of Algorithm \ref{a:dynamicfactor} which utilizes Lemma \ref{disjdyn} instead of Lemma \ref{l:nondisjointrecurrence}, using (ordered) lists instead of sets, and which according to Corollary \ref{resultislex} computes all of its result lists in lexicographic order.

\begin{alg}\label{a:dynlist}\
\begin{algorithmic}[1]
\Function{LexicographicFactorizationListsUpToElement}{$n, (g_1,...,g_d)$}
\State $F[0] \gets [\mathbf 0]$
\ForAll{$m \in [0, n]$}
	\State $Z \gets [\ ]$
	\ForAll{$i = 1, 2, \ldots, d$ with $m - g_i \ge 0$}
		\State $Z \gets Z \sqcup \left[\incr_i(\aa)  : \aa \in F[m - g_i]\ \mathsf{if}\ \mathsf{isAllZeroesLeftOfIndex}(\aa, i)\right]$
	\EndFor
	\State $F[m] \gets Z$
\EndFor
\State \Return $F$
\EndFunction
\end{algorithmic}
\end{alg}

\subsection{Parallelism}
The above Algorithm \ref{a:dynlist} can be parallelized either elementwise (in line 3), or factorizationwise (in line 6), or \emph{both}.\footnote{Indeed, each element could be assigned to a CPU thread, and each thread could launch its own factorizationwise GPU kernels.}

In elementwise (threadwise) parallelism, the computation of $\Zz(0), \Zz(1),...,\Zz(n)$ is split up into batches of size $b$, which is at most $g_{smallest}$. If all factorization lists up to $\Zz(m)$ are in the memo, then each of $\Zz(m + 1), ..., \Zz(m + b)$ can be computed in parallel as they do not affect each other's outputs.

Alternatively, in factorizationwise (instructionwise) parallelism, for a fixed $m$ and $i$, the copy-and-increment step of computing $\{\incr_i(\aa)  : \aa \in F[m - g_i]\}$ can be parallelized among as many workers as there are factorizations. We choose to implement this variant presently.

\subsection{Saving subset cardinalities}
Suppose we are to implement the above factorizationwise parallelism as a GPU kernel. As multiple threads cannot write to the same block of memory in a kernel (undefined behavior, as one would overwrite the other), kernels typically operate by assigning each worker a distinct, and notably \textit{pre-assigned}, output location upon which to write. If we want the result of the copies to be contiguous in the target memory, rather than having gaps anywhere that \textsf{isAllZeroesLeftOfIndex} was false, then it may be useful to save the cardinality of each $\Zz_{=i}(n)$ as they are computed, where 
$$\Zz_{=i}(n) = \left\{\vec a \in \Zz(n) \mid {a_j = 0\ \forall j < i \mathsf{\ and\ } a_i>0}\right\}$$
We note that $\Zz_{\ge i}(n) = \bigsqcup_{j \ge i} \Zz_{=i}(n)$,\footnote{Except $n=0$, which has $\abs{\Zz_{\ge i}(0)}=1$ but $\abs{\Zz_{=i}(0)}=0$ for all $i$. Thus some may prefer to write $\Zz_{\not<i}(n)$.} and that the beginning index of $\Zz_{\ge i}(n)$ inside $\Zz(n)$ is precisely $\sum_{j < i} \abs{\Zz_{=i}(n)}$. We further note that each of the $i$-indexed expressions on the right of Lemma \ref{disjlex} coincides exactly with $\Zz_{=i}(n)$. That is, $ \Zz_{=i}(n) = \incr_i\left(\Zz_{\ge i}(n-g_i)\right) $. And thus, the cardinality of each $\Zz_{=i}(n)$ can be saved by simply recording how many results are copy-and-incremented during the $i$'th iteration of the loop of calculating $\Zz(n)$, which is itself $\sum_{j \ge i} \abs{\Zz_{=j}(n-g_i)}$.

\subsection{Factorizationwise parallel algorithm}
We now present the natural variant of Algorithm \ref{a:dynlist} which uses factorizationwise parallelism and saves the cardinality of each $\Zz_{=i}(n)$ in order to precalculate the index from which to begin the copy. For simplicity we allow the kernel call to only use 1 thread per block, while our actual implementation uses more. The code references a kernel \textsf{copyAndIncrementIndex(Factorization* destination, Factorization* source, int index)}.

\begin{alg}\label{a:dynlistparallel}\
\begin{algorithmic}[1]
\Function{LexFacListsUpToElement\_FactorizationwiseParallel}{$n, (g_1,...,g_d)$}
\State $Z[0] \gets [\mathbf 0]$
\State \textsf{int[d][n] cardinalities};
\State \textsf{cardinalities[0] = [1,...,1]};
\ForAll{$m \in [0, n]$}
	% \State $Z \gets [\ ]$
    \State $\mathsf{cardinalityThisElement=0}$;
	\ForAll{$i = 1, 2, \ldots, d$ with $m - g_i \ge 0$}
        \State $\mathsf{startIndex} = \sum_{j < i} \mathsf{cardinalities}[m-g_i][j]$;
        \State $\mathsf{cardinalityThisIndex} = \sum_{j \ge i} \mathsf{cardinalities}[m-g_i][j]$;
        \State $\mathsf{copyAndIncrementIndex\kernelcall{cardinalityThisIndex, 1}(}$
        \State \hspace{\algorithmicindent} $Z[m][\mathsf{cardinalityThisElement}],\ Z[m-g_i\mathsf{][startIndex],\ i)}$;
        \State $\mathsf{cardinalities}[m][i]=\mathsf{cardinalityThisIndex}$;
        \State $\mathsf{cardinalityThisElement}\ +=\ \mathsf{cardinalityThisIndex}$;
	\EndFor
\EndFor
\State \Return $Z$
\EndFunction
\end{algorithmic}
\end{alg}

Our reference implementation runs both the parallel and single-threaded variants of memo population and records the times of both for comparison.

\section{Parallel bounded lexicographic enumeration}
Parallel bounded lexicographic enumeration as described in \cite[Algorithm 3.1]{parallelboundedlexicographic} splits work among many threads on a GPU and calls \textsf{nextCandidate} as a kernel which runs on each work slice.
We encourage the reader to review \cite{parallelboundedlexicographic} for a complete description of the behavior of the \textsf{nextCandidate} function and for an introduction to GPU programming generally.

\pagebreak
For reference, we reproduce here in full the algorithm \textsf{nextCandidate} as it originally appeared before making further modifications. $\phi(\vec a) = \sum_i a_ig_i$.

\begin{alg}[The lexicographic greatest next-candidate algorithm]\ \\
$\mathsf{nextCandidate}:\ ParallelLexicographicState \to ParallelLexicographicState$\\
Input: $(n,\ (g_1,...,g_d),\ (a_1,...,a_d),\ (b_1,...,b_d),\ wasValid,\ endOfStream)$\\
Output: The input variables, modified (the function operates in-place)
\begin{enumerate}
    \item $\mathsf{if}\ (endOfStream)\ return;$
   \item $\mathsf{let}\ i=0;\ \mathsf{for}\ (\mathsf{let}\ j = 1;\ j\le dim;\ j++) \{\ \mathsf{if}\ (a_j > 0)\ i = j\ \}$. \\(Finding the rightmost nonzero index excluding the final.)
   \item $\mathsf{if}\ (i=0) \{\ endOfStream=true;\ return;\ \}$, detecting if the last candidate was the final
    \item $a_d = 0$, resetting the final index to zero in case it was nonzero from having been solved for previously;
   \item $a_i --$ , which is the first step of $\mathsf{decrementAndSolve}(\vec a, i)$, consisting of steps 5-11
   \item $\mathsf{let}\ p = n - \phi(\vec a)$, the element to attempt to factor by $g_{i+1}$
   \item $\mathsf{let}\ m = p / g_{i+1}$, where division denotes integer division with remainder discarded.
   \item $\mathsf{let}\ r = p - m*g_{i+1}$, the remainder;
   \item $wasValid = true$, temporarily;
   \item $\mathsf{if}\ (r \neq 0) \{\ m ++;\ \mathsf{wasValid}=false; \} $
   \item $a_{i+1} = m$
   \item $\mathsf{if}\ ((a_1,...,a_d) \leq_{lex} (b_1,...,b_d))\ \{\ endOfStream = true\ \}$
   \item Return $(n,\ (g_1,...,g_d),\ (a_1,...,a_d),\ (b_1,...,b_d),\ wasValid,\ endOfStream)$
\end{enumerate}
\end{alg}

\subsection{Adding dynamic behavior}
As mentioned in the original work, this method admits a dynamic optimization. 
\begin{quote}{\cite{parallelboundedlexicographic}, \quotes{Future work}}
In $\mathsf{nextCandidate}$, if e.g. $i=d-3$, after decrementing $a_i$ and zeroing $a_d$, all of the remaining factorizations $\vec c$ of the form $(a_1,...,a_i,c_{d-2},c_{d-1},c_d)$ are in bijection with $\Zz(n-\phi(\vec a), (g_{d-1},g_{d-1},g_d))$. If $\Zz(x, (g_{d-2},g_{d-1},g_{d}))$ is stored for all $0 \le x \le m$ for some small value of $m$, and these sets are made accessible to each worker during their $\mathsf{nextCandidate}$ evaluation, then when index $i=d-3$ and $\phi(\vec a) \le m$, the remaining solutions $\vec c$ with $c_j = a_j$ for all $1 \le j \le d-3$ could be returned all at once by copying $\Zz(n-\phi(\vec a), (g_{d-2},g_{d-1},g_{d}))$ and prepending $(a_1,...,a_i)$ to each. (The signature and behavior of the $\mathsf{nextCandidate}$ function and its calling code would need to be modified appropriately to allow an output of more than one result.)

We give the above example of dynamic behavior for $i=d-3$, but it is valid for any number of the trailing indices.
\end{quote}

Happily, we just discussed in Section 2 a nice way to compute all factorization sets \linebreak$\Zz(0, (g_{d-\mathsf{memoDim}+1},...,g_{d})), ..., \Zz(n, (g_{d-\mathsf{memoDim}+1},...,g_{d}))$.

We add the observation that, after copying $\Zz(n-\phi(\vec a), (g_{d-\mathsf{memoDim}+1},...,g_{d}))$ and prepending $(a_1,...,a_{d-\mathsf{memoDim}})$ to each, it is not necessary to set the state's \textsf{lastCandidate} to the actually lexicographically final factorization returned. It is sufficient to leave the value of $\vec a$ at \linebreak$(a_1,...,a_{d-\mathsf{memoDim}}, 0,...,0)$: for this is lexicographically less than every valid candidate that was just copied; and there are no factorizations skipped in between the final copied factorization and $\vec a$ (because such a factorization $\vec f$ would have $f_i = a_i$ for all $i \le d-\mathsf{memoDim}$, which would imply it was just copied).

\pagebreak
\begin{alg}[Lexicographic next-candidate with low-dimension tabulation]\ \\
$\mathsf{nextCandidateDynamic}:\ ParallelDynamicLexState \to ParallelDynamicLexState$\\
Input: $(n,\ (g_1,...,g_d),\ (a_1,...,a_d),\ (b_1,...,b_d),\ \mathsf{wasValid},\ \mathsf{endOfStream},\ $
\\\hspace*{\parindent * 3}$\mathsf{Memo},\ \mathsf{memoDim},\ \mathsf{topOfMemo},\ \mathsf{Outputs})$\\
Output: The input variables, modified (the function operates in-place)
\begin{enumerate}
    \item $\mathsf{if}\ (\mathsf{endOfStream})\ return;$
   \item $\mathsf{let}\ i=0;\ \mathsf{for}\ (\mathsf{let}\ j = 1;\ j\le d;\ j++) \{\ \mathsf{if}\ (a_j > 0)\ i = j\ \}$. \\(Finding the rightmost nonzero index excluding the final.)
   \item $\mathsf{if}\ (i=0) \{\ \mathsf{endOfStream}=true;\ return;\ \}$
    \item $a_d = 0$
   \item $a_i --$
   \item If $i = d-\mathsf{memoDim}$ and $p < \mathsf{topOfMemo}$:
   \begin{enumerate}
       \item for each index $j$ of $\mathsf{Memo}[p]$, copy $\mathsf{Memo}[p][j]$ onto the final $\mathsf{memoDim}$ coordinates of $\mathsf{Outputs}[j]$, and then copy $(a_1,...,a_{d-\mathsf{memoDim}})$ onto the leading coordinates. Also, set $\mathsf{wasValid}=false$.
   \end{enumerate}
   \item Else:
   \begin{enumerate}
       \item $\mathsf{let}\ p = n - \phi(\vec a)$
       \item $\mathsf{let}\ m = p / g_{i+1}$
       \item $\mathsf{let}\ r = p - m*g_{i+1}$
       \item $\mathsf{wasValid} = true$
       \item $\mathsf{if}\ (r \neq 0) \{\ m ++;\ \mathsf{wasValid}=false; \} $
   \item $a_{i+1} = m$
   \end{enumerate}
   \item $\mathsf{if}\ ((a_1,...,a_d) \leq_{lex} (b_1,...,b_d))\ \{\ \mathsf{endOfStream} = true\ \}$
   \item Return.
\end{enumerate}
\end{alg}

In the same way that the original \textsf{nextCandidate} needs to be followed by copying each state's \textsf{lastCandidate} (i.e. $\vec a$) to the state's \textsf{Buffer} after each iteration, similarly \textsf{nextCandidateDynamic} needs to also have its \textsf{Outputs} dumped to the \textsf{Buffer} and zeroed after each kernel call. As the function has no return value, the final saving of factorizations occurs inside \textsf{copyDeviceBufferToHostAndClear}.

Analogous to \cite[Algorithm 5.1]{parallelboundedlexicographic}, we give an abbreviated description of the overall algorithm to call \textsf{nextCandidateDynamic} in a loop until completion, and we encourage the reader to review the full implementation in the supplementary code files.

\begin{alg}\label{a:dynlistparallel}\
\begin{algorithmic}[1]
\State \textsf{function FactorizationsParallelLexicographicWithMemoization(}$n, (g_1,...,g_d)$\textsf{) \{}
% \Function{FactorizationsParallelLexicographicWithMemoization}{$n, (g_1,...,g_d), \mathsf{memoDim}$}
\State  \hspace{\algorithmicindent} \textsf{memo, states, buffer, outputs = allocateEverything();}
\State  \hspace{\algorithmicindent} \textsf{populateMemo(memo)};
\State  \hspace{\algorithmicindent} \textsf{setInitialCandidate(states)};
\State  \hspace{\algorithmicindent} \textsf{iteration = 0};
\State  \hspace{\algorithmicindent} \textsf{while (!allStatesAreEndOfStream(states)) \{}
\State  \hspace{\algorithmicindent} \hspace{\algorithmicindent} \textsf{if (iteration \% kernelsBetweenFreshBounds = 0) \{ splitWork(states); \}}
\State  \hspace{\algorithmicindent} \hspace{\algorithmicindent} \textsf{nextCandidateDynamic}$\kernelcall{\mathsf{BLOCKS, THREADS}}$\textsf{(states, memo, outputs);}
\State  \hspace{\algorithmicindent} \hspace{\algorithmicindent} \textsf{copyOutputsToBufferAndClear}\textsf{(states, outputs, buffer);}
\State  \hspace{\algorithmicindent} \hspace{\algorithmicindent} \textsf{if (anyBufferIsFull(states)) \{ copyDeviceBufferToHostAndClear(states); \}}
\State  \hspace{\algorithmicindent} \hspace{\algorithmicindent} \textsf{iteration++};
\State  \hspace{\algorithmicindent} \textsf{\}}
\State  \hspace{\algorithmicindent}  \textsf{copyDeviceBufferToHostAndClear(buffer);}
\State \textsf{\}}
% \EndFunction
\end{algorithmic}
\end{alg}

\section{CUDA implementations}
In the supplementary code files, we update the CUDA implementation of parallel bounded lexicographic streams from \cite{parallelboundedlexicographic} to include the above described algorithms. This entails the lexicographic variant of the dynamic factorizations algorithm, in both its single-threaded CPU form \textsf{LexicographicFactorizationListsUpToElement} and its parallel form 
% \linebreak
\textsf{LexFacListsUpToElement\_FactorizationwiseParallel}; the lexicographic successor algorithm in its dynamic variant \textsf{nextCandidateDynamic}; and the complete parallel dynamic lexicographic factorizations algorithm \textsf{FactorizationsParallelLexicographicWithMemoization}.

\subsection{Performance}

Table 1 shows the runtimes of the parallel lexicographic algorithm with low-dimension tabulation, including both the times to populate the memo on the CPU and the GPU. These runtimes compare favorably to those recorded in \cite{parallelboundedlexicographic} for the same generators.

CPU memo population time is consistently lower than GPU for small memo sizes, but GPU memo population is faster after some threshold. We note that some cases have lower memo dimensions resulting in faster overall runtimes than higher memo dimensions for the same element, e.g. dimension 6 with memo dimension 3 vs 4, while others are faster with higher memo dimension, e.g. dimension 9 with memo dimension 4 vs 5.
% \pagebreak

\begin{table}
{\footnotesize

\begin{tabular}{lll|llll}
\hline
dim & memo\_dim & element & num\_results & cpu\_memo\_us & gpu\_memo\_us & runtime\_ms \\
\hline
3   & 1         & 100000  & 273793       & 313           & 9764          & 2733        \\
3   & 1         & 200000  & 1094693      & 452           & 9205          & 7904        \\
3   & 1         & 300000  & 2462699      & 322           & 9087          & 16241       \\
\hline
4   & 2         & 5000    & 29601        & 2920          & 27188         & 400         \\
4   & 2         & 10000   & 232364       & 2888          & 27229         & 420         \\
4   & 2         & 15000   & 779252       & 3073          & 33543         & 679         \\
\hline
5   & 2         & 1000    & 1920         & 1674          & 30028         & 485         \\
5   & 2         & 3000    & 125780       & 1713          & 29153         & 1047        \\
5   & 2         & 5000    & 928872       & 1943          & 30699         & 2101        \\
5   & 2         & 10000   & 14375415     & 1624          & 33645         & 7600        \\
\hline
5   & 3         & 1000    & 1920         & 7952          & 26683         & 502         \\
5   & 3         & 3000    & 125780       & 8295          & 28300         & 702         \\
5   & 3         & 5000    & 928872       & 7951          & 28841         & 953         \\
\hline
6   & 3         & 1000    & 10873        & 7586          & 27999         & 601         \\
6   & 3         & 3000    & 1910535      & 8255          & 27061         & 1751        \\
6   & 3         & 5000    & 22955008     & 7341          & 26576         & 6298        \\
\hline
6   & 4         & 1000    & 10873        & 43686         & 24625         & 1155        \\
6   & 4         & 2000    & 273487       & 42812         & 25625         & 4079        \\
6   & 4         & 3000    & 1910535      & 45837         & 26152         & 10640       \\
\hline
7   & 4         & 1000    & 52036        & 44906         & 24383         & 1153        \\
7   & 4         & 1500    & 473670       & 45228         & 28812         & 2299        \\
7   & 4         & 2000    & 2369185      & 55714         & 26664         & 5416        \\
\hline
8   & 4         & 1000    & 216822       & 47843         & 25655         & 1659        \\
8   & 4         & 1500    & 2740729      & 48720         & 26556         & 5736        \\
8   & 4         & 2000    & 17552389     & 46092         & 28807         & 25165       \\
\hline
9   & 4         & 500     & 9054         & 49239         & 29815         & 1326        \\
9   & 4         & 1000    & 801745       & 47665         & 26707         & 8202        \\
9   & 4         & 1500    & 13936185     & 49328         & 26591         & 52694       \\
\hline
9   & 5         & 500     & 9054         & 21401         & 18208         & 1129        \\
9   & 5         & 1000    & 801745       & 22461         & 16717         & 6018        \\
9   & 5         & 1500    & 13936185     & 22711         & 18628         & 47513       
\end{tabular}
}
\medskip
\caption{\footnotesize Runtimes for computing $\mathsf Z(n, (13, 37, 38 [, 40,41,..., 45]))$ using parallel bounded lexicographic enumeration with low-dimension tabulation. Overall runtimes are in milliseconds while memo population times are in microseconds. Computations were performed on a machine with an AMD Ryzen 3900X CPU and an NVIDIA RTX 3080 GPU. CUDA implementation parameters were $68*2$ grid size (the 3080 has 68 SMs), 8 threads per block, 40000 buffer per thread (note that this is a significant increase from the previous implementation's 1000), and 1024 \textsf{kernelsBetweenNewBounds}. All runs were performed with a fully populated memo, i.e., populated up to at least the highest element tested for that group (excepting dimension 3).
% \linebreak \linebreak
}
\end{table}

\section{Future work}
    Describe the asymptotic runtime behavior of \textsf{FactorizationsParallelLexicographicWithMemoization} and \textsf{LexFacListsUpToElement\_FactorizationwiseParallel}.

% %%%%%%%%%%%%%%%%%%%%%%%%%%%%%%%%%%%%%%%%%%%%%%%%%%%%%%%%%%%%%%%%%%%%%%%%%
% \section*{Acknowledgements}%%%%%%%%%%%%%%%%%%%%%%%%%%%%%%%%%%%%%%%%%%%%%%
% %raggedbottom%%%%%%%%%%%%%%%%%%%%%%%%%%%%%%%%%%%%%%%%%%%%%%%%%%%%%%%%%%%%

% 

%%%%%%%%%%%%%%%%%%%%%%%%%%%%%%%%%%%%%%%%%%%%%%%%%%%%%%%%%%%%%%%%%%%%%%%%%
%%%%%%%%%%%%%%%%%%%%%%%%%%%%%%%%%%%%%%%%%%%%%%%%%%%%
%%%%%%%%%%%%%%%%%%%%%%%%%%%%%%%%%%%%%%%%%%%%%%%%%%%%%%%%%%%%%%%%%%%%%%%%%

%%%%%%%%%%%%%%%%%%%%%%%%%%%%%%%%%%%%%%%%%%%%%%%%%%%%%%%%%%%%%%%%%%%%%%%%%
\end{document}